\documentclass[12pt]{amsart}
\usepackage{amsmath,amssymb,amsbsy,amsfonts, amsthm,
latexsym,amsopn,amstext,
amsopn,amstext,amsxtra,euscript,amscd}
\begin{document}
\newfont{\teneufm}{eufm10}
\newfont{\seveneufm}{eufm7}
\newfont{\fiveeufm}{eufm5}
%
%
\newfam\eufmfam
                                  \textfont\eufmfam=\teneufm
\scriptfont\eufmfam=\seveneufm
                                  \scriptscriptfont\eufmfam=\fiveeufm

%
%
\def\frak#1{{\fam\eufmfam\relax#1}}
%


\def\rad{{\rm rad}}

\def\bbbr{{\rm I\!R}} 
\def\bbbc{{\rm I\!C}} 
\def\bbbm{{\rm I\!M}}
\def\bbbn{{\rm I\!N}} 
\def\bbbf{{\rm I\!F}}
\def\bbbh{{\rm I\!H}}
\def\bbbk{{\rm I\!K}}
\def\bbbl{{\rm I\!L}}
\def\bbbp{{\rm I\!P}}
\def\bbbq{{\rm I\!Q}}
\newcommand{\lcm}{{\rm lcm}}
\def\bbbone{{\mathchoice {\rm 1\mskip-4mu l} {\rm 1\mskip-4mu l}
{\rm 1\mskip-4.5mu l} {\rm 1\mskip-5mu l}}}
\def\bbbc{{\mathchoice {\setbox0=\hbox{$\displaystyle\rm C$}\hbox{\hbox
to0pt{\kern0.4\wd0\vrule height0.9\ht0\hss}\box0}}
{\setbox0=\hbox{$\textstyle\rm C$}\hbox{\hbox
to0pt{\kern0.4\wd0\vrule height0.9\ht0\hss}\box0}}
{\setbox0=\hbox{$\scriptstyle\rm C$}\hbox{\hbox
to0pt{\kern0.4\wd0\vrule height0.9\ht0\hss}\box0}}
{\setbox0=\hbox{$\scriptscriptstyle\rm C$}\hbox{\hbox
to0pt{\kern0.4\wd0\vrule height0.9\ht0\hss}\box0}}}}
\def\bbbq{{\mathchoice {\setbox0=\hbox{$\displaystyle\rm
Q$}\hbox{\raise
0.15\ht0\hbox to0pt{\kern0.4\wd0\vrule height0.8\ht0\hss}\box0}}
{\setbox0=\hbox{$\textstyle\rm Q$}\hbox{\raise
0.15\ht0\hbox to0pt{\kern0.4\wd0\vrule height0.8\ht0\hss}\box0}}
{\setbox0=\hbox{$\scriptstyle\rm Q$}\hbox{\raise
0.15\ht0\hbox to0pt{\kern0.4\wd0\vrule height0.7\ht0\hss}\box0}}
{\setbox0=\hbox{$\scriptscriptstyle\rm Q$}\hbox{\raise
0.15\ht0\hbox to0pt{\kern0.4\wd0\vrule height0.7\ht0\hss}\box0}}}}
\def\bbbt{{\mathchoice {\setbox0=\hbox{$\displaystyle\rm
T$}\hbox{\hbox to0pt{\kern0.3\wd0\vrule height0.9\ht0\hss}\box0}}
{\setbox0=\hbox{$\textstyle\rm T$}\hbox{\hbox
to0pt{\kern0.3\wd0\vrule height0.9\ht0\hss}\box0}}
{\setbox0=\hbox{$\scriptstyle\rm T$}\hbox{\hbox
to0pt{\kern0.3\wd0\vrule height0.9\ht0\hss}\box0}}
{\setbox0=\hbox{$\scriptscriptstyle\rm T$}\hbox{\hbox
to0pt{\kern0.3\wd0\vrule height0.9\ht0\hss}\box0}}}}
\def\bbbs{{\mathchoice
{\setbox0=\hbox{$\displaystyle     \rm S$}\hbox{\raise0.5\ht0\hbox
to0pt{\kern0.35\wd0\vrule height0.45\ht0\hss}\hbox
to0pt{\kern0.55\wd0\vrule height0.5\ht0\hss}\box0}}
{\setbox0=\hbox{$\textstyle        \rm S$}\hbox{\raise0.5\ht0\hbox
to0pt{\kern0.35\wd0\vrule height0.45\ht0\hss}\hbox
to0pt{\kern0.55\wd0\vrule height0.5\ht0\hss}\box0}}
{\setbox0=\hbox{$\scriptstyle      \rm S$}\hbox{\raise0.5\ht0\hbox
to0pt{\kern0.35\wd0\vrule height0.45\ht0\hss}\raise0.05\ht0\hbox
to0pt{\kern0.5\wd0\vrule height0.45\ht0\hss}\box0}}
{\setbox0=\hbox{$\scriptscriptstyle\rm S$}\hbox{\raise0.5\ht0\hbox
to0pt{\kern0.4\wd0\vrule height0.45\ht0\hss}\raise0.05\ht0\hbox
to0pt{\kern0.55\wd0\vrule height0.45\ht0\hss}\box0}}}}
\def\bbbz{{\mathchoice {\hbox{$\sf\textstyle Z\kern-0.4em Z$}}
{\hbox{$\sf\textstyle Z\kern-0.4em Z$}}
{\hbox{$\sf\scriptstyle Z\kern-0.3em Z$}}
{\hbox{$\sf\scriptscriptstyle Z\kern-0.2em Z$}}}}
\def\ts{\thinspace}

\newtheorem{theorem}{Theorem}
\newtheorem{lemma}[theorem]{Lemma}
\newtheorem{claim}[theorem]{Claim}
\newtheorem{cor}[theorem]{Corollary}
\newtheorem{prop}[theorem]{Proposition}
\newtheorem{definition}{Definition}
\newtheorem{question}[theorem]{Open Question}

\def\squareforqed{\hbox{\rlap{$\sqcap$}$\sqcup$}}
\def\qed{\ifmmode\squareforqed\else{\unskip\nobreak\hfil
\penalty50\hskip1em\null\nobreak\hfil\squareforqed
\parfillskip=0pt\finalhyphendemerits=0\endgraf}\fi}

\def\cA{{\mathcal A}}
\def\cB{{\mathcal B}}
\def\cC{{\mathcal C}}
\def\cD{{\mathcal D}}
\def\cE{{\mathcal E}}
\def\cF{{\mathcal F}}
\def\cG{{\mathcal G}}
\def\cH{{\mathcal H}}
\def\cI{{\mathcal I}}
\def\cJ{{\mathcal J}}
\def\cK{{\mathcal K}}
\def\cL{{\mathcal L}}
\def\cM{{\mathcal M}}
\def\cN{{\mathcal N}}
\def\cO{{\mathcal O}}
\def\cP{{\mathcal P}}
\def\cQ{{\mathcal Q}}
\def\cR{{\mathcal R}}
\def\cS{{\mathcal S}}
\def\cT{{\mathcal T}}
\def\cU{{\mathcal U}}
\def\cV{{\mathcal V}}
\def\cW{{\mathcal W}}
\def\cX{{\mathcal X}}
\def\cY{{\mathcal Y}}
\def\cZ{{\mathcal Z}}

\newcommand{\comm}[1]{\marginpar{%
\vskip-\baselineskip 
\raggedright\footnotesize
\itshape\hrule\smallskip#1\par\smallskip\hrule}}





\def\ve{\varepsilon}

\hyphenation{re-pub-lished}

\def\ord{{\mathrm{ord}}}
\def\Nm{{\mathrm{Nm}}}
\renewcommand{\vec}[1]{\mathbf{#1}}

\def \F{{\mathbb{F}}}
\def \L{{\bbbl}}
\def \K{{\bbbk}}
\def \Z{{\mathbb{Z}}}
\def \N{{\bbbn}}
\def \Q{{\mathbb{Q}}}
\def\E{{\mathbf E}}
\def\bH{{\mathbf H}}
\def\G{{\mathcal G}}
\def\O{{\mathcal O}}
\def\cS{{\mathcal S}}
\def \R{{\bbbr}}
\def\Fp{\F_p}
\def \fp{\Fp^*}
\def\\{\cr}
\def\({\left(}
\def\){\right)}
\def\fl#1{\left\lfloor#1\right\rfloor}
\def\rf#1{\left\lceil#1\right\rceil}

\def\Zm{\Z_m}
\def\Zt{\Z_t}
\def\Zp{\Z_p}
\def\Um{\cU_m}
\def\Ut{\cU_t}
\def\Up{\cU_p}

\def\ep{{\mathbf{e}}_p}
\def\eq{{\mathbf{e}}_q}
\def\eMx{{\mathbf{e}}_{M_f(x)}}
\def\eMxd{{\mathbf{e}}_{M_f(x)/d}}

\def\e{{\mathbf{e}}}

\def \Prob{{\mathrm {}}}

\def\LC{{\cL}_{C,\cF}(Q)}
\def\LCn{{\cL}_{C,\cF}(nG)}
\def\Mrs{\cM_{r,s}\(\F_p\)}

\def\Fbar{\overline{\F}_q}
\def\Fn{\F_{q^n}}
\def\Px{\cP_f(x)}

\def \hatf{\widehat{f}}

\def\mand{\qquad \mbox{and} \qquad}

\def\MOV{{\bf{MOV}}}

\title{On Pseudopoints of Algebraic Curves}

\author{Reza R. Farashahi and Igor E.~Shparlinski}
\address{
Department of Computing, Faculty of Science, Macquarie University, Sydney, NSW 2109,
Australia}
\email{\{reza,igor\}@comp.mq.edu.au}

\date{\today}

\begin{abstract}
Following Kraitchik and Lehmer, we say that
a positive integer $n\equiv1\pmod 8$ is an $x$-pseudosquare if it
is a quadratic residue for each odd prime $p\le x$, yet is not a square.
We extend this defintion to algebraic curves and say that $n$ is
an $x$-pseudopoint of a curve $f(u,v) = 0$ (where $f \in \Z[U,V]$)
if for all sufficiently large primes $p \le x$ the congruence
$f(n,m)\equiv 0 \pmod p$ is satisfied for some $m$.

We use the Bombieri bound of exponential sums along a curve to
estimate the smallest  $x$-pseudopoint, which shows the limitations
of the modular approach to searching for points on curves.
\end{abstract}

\maketitle


\section{Introduction}

Following Lehmer in~\cite{L}, given a real $x \ge 1$,
we say that a nonsquare positive integer
$n$ is  an {\it $x$-pseudosquare\/} if $n\equiv 1\pmod 8$
and $(n/p)=1$ for each
odd prime $p\le x$, see also~\cite{PoSh,Schin,Sor,WW}
for further results.
Here we generalise this notion and
introduce and study {\it $x$-pseudopoints\/} on algebraic curves.

More precisely, given an absolutely irreducible polynomial $f(U,V)
\in \Z[U,V]$ and an integer $q\ge 1$ we denote
\begin{equation}
\label{eq:Zfq}
\cZ_f(q) = \{ (n,m)\ : \ 0\le n, m < q,\ f(n,m) \equiv 0 \pmod {q}\}.
\end{equation}

Then, we define $\cP_f$ as the set of primes $p$ for which
$\cZ_f(p)$ is not empty. We note that an absolutely irreducible
polynomial $f$ remains absolutely
irreducible modulo all sufficiently large prime numbers $p$ by
Ostrowski~\cite{Ost}. Therefore, by the Weil bound,
see~\cite[Section~VIII.5, Bound~(5.7)]{Lor}, we conclude that
$\cP_f$ contains all sufficiently large primes. In particular,
for
$$
M_f(x) = \prod_{p \in \Px} p
$$
by the prime number theorem, we have
\begin{equation}
\label{eq:Mx}
 M_f(x) = \exp((1 + o(1)) x).
\end{equation}

Furthermore for a real $x \ge 1$,
we denote by $\Px = \cP_f \cap [2,x]$ and  say  that an
integer $n\ge 0$ is an $x$-pseudopoint of
$f$ if for all $p \in \Px$ we have $(n,m)\in \cZ_f(p)$ for some $m$,
but the equation $f(n,m) = 0$ has no integer solution $m \in \Z$.
We note that Bernstein~\cite{Bern} has introduced and studied
this notion in   the case of the polynomials of the form
$f(U,V) = g(U)-V^2$.

Clearly, apart of the congruence condition $n\equiv 1\pmod 8$
and the coprimality condition  $\gcd(n,p) =1$
for primes $p\le x$,
the polynomial $f(U,V) = U-V^2 $ corresponds to the case of
$x$-pseudosquares.

It is easy  to
show that for any absolutely irreducible polynomial $f(U,V) \in \Z[U,V]$
 nonlinear in $V$, that is, $\deg_V f \ge 2$,
the smallest
$x$-pseudopoint $N_f(x)$ satisfies the bound
\begin{equation}
\label{eq:TrivBound}
N_f(x) = O(M_f(x)) = \exp((1 + o(1)) x).
\end{equation}
Clearly the condition of nonlinearity in $V$ is necessary (for
example, the polynomial $f(U,V) = U^2 - V$ does not have any pseudopoints).

Indeed, the bound~\eqref{eq:TrivBound} can be  derived
from the Chinese remaindering theorem combined with
the Weil bound (see~\cite[Section~VIII.5, Bound~(5.7)]{Lor}),
and also the bound of Bombieri and
Pila~\cite{Pila} on the number of integer points on
plane curves.

Here we use the Bombieri bound~\cite{Bomb} to
improve~\eqref{eq:TrivBound}.

\begin{theorem}
\label{thm:Distr}
For any absolutely irreducible polynomial
$$f(U,V) \in \Z[U,V]$$
that is nonlinear in $V$, that is,
$$
\deg_V f \ge 2,
$$  we have
$$
N_f(x) \le   M_f(x)^{1/2+o(1)}.
$$
\end{theorem}

The bound of Theorem~\ref{thm:Distr} is an analogues of
similar, albeit stronger, estimates for pseudosquares,
see~\cite{PoSh,Schin}. Besides it shows the limitations of
the modular approach to searching for points on curves.
Indeed, assuming that $\cP_f$ consists of all primes
(otherwise the equation $f(n,m) = 0$ has no integer solutions),
we see that there is a reasonably small point which
is a solution to the corresponding congruence modulo
all small primes but is not a solution to the equation.

\section{Preparations}

We need some background on exponential sums and congruences.

For an integer $q$ and a complex $z$, we denote
$$
\eq(z) = \exp(2 \pi i z/q)
$$
and recall the identity
\begin{equation}
\label{eq:Ident}
\frac{1}{q}\sum_{-q/2 < a \le q/2}  \eq(an) =
\left\{\begin{array}{ll}
1,&\quad\text{if $n\equiv 0 \pmod q$,}\\
0,&\quad\text{if $n\not\equiv 0 \pmod q$,}
\end{array}
\right.
\end{equation}
which follows from the formula for the sum of geometric
progression.

We also need  the following bound
\begin{equation}
\label{eq:Incompl}
    \sum_{n=1}^{N} \eq(an) = O\( \min\{N, q/|a|\}\),
\end{equation}
which holds for any integers  $a$ and $N\ge 1$ with $0 < |a| \le q/2$,
see~\cite[Bound~(8.6)]{IwKow}.

Our main tool is the following special case of
the Bombieri bound~\cite[Theorem~6]{Bomb}
of exponential sums along a curve.

\begin{lemma}\label{lem:Bomb}
Assume that for a prime $p$, a polynomial
$f(U,V) \in \Z[U,V]$ is such that its reduction
modulo $p$ does not have a
factor of the form $U-\alpha$ with some
$\alpha \in \F_p$. Then uniformly over $a\in \Z$ with
$\gcd(a,p)=1$
$$  \sum_{\substack{(u,v) \in \cZ_f(p)}} \ep(au)
= O(p^{1/2}),
$$
where the implied constant depends only on $\deg f$.
\end{lemma}

Finally, we need the following consequence of the Chinese
remainder theorem (see also~\cite[Equation~(12.21)]{IwKow}
for a very similar statement).

\begin{lemma}\label{lem:CRT}
For any polynomial
$$f(U,V) \in \Z[U,V],
$$
we have
$$
\sum_{\substack{ (u,v) \in \cZ_f(M_f(x))}} \eMx(au)
=
\prod_{p\in \cP_f(x)}\ \sum_{\substack{(u,v) \in \cZ_f(p)}} \ep(au).
$$
\end{lemma}

\begin{proof}
Let $\Z_m$ denote the residue ring modulo $m$.
 From the Chinese remainder theorem, there is a bijection
\begin{equation}
\label{eq:CRT}
\Z_{M_f(x)} \simeq \ \bigotimes_{p\in \cP_f(x)}\Z_p
\end{equation}
by  $u \mapsto (u_p)_{p\in \cP_f(x)}$. On the other hand, every tuple
$(u_p)_{p\in \cP_f(x)}$ corresponded to the unique element
$$
u=\sum_{p\in \cP_f(x)} \frac{M_f(x)}{p} u_p \in \Z_{M_f(x)}.
$$
Then,
$$
\eMx(au)=\prod_{p\in \cP_f(x)} \ep(au_p).
$$
Moreover,  in a natural way~\eqref{eq:CRT} yields
a bijection between the points $(u,v)$ in
$\cZ_f(M_f(x))$ and the tuples of points
$$((u_p,v_p))_{p\in \cP_f(x)}\in \bigotimes_{p\in \cP_f(x)} \cZ_f(p).
$$ Therefore,
$$
\begin{aligned}
\sum_{(u,v) \in \cZ_f(M_f(x))} \eMx(au)
&=
\sum_{ \substack{((u_p,v_p))_{p\in \cP_f(x)} \\ (u_p,v_p) \in
\cZ_f(p)}}\ \prod_{p\in \cP_f(x)} \ep(au_p)\\
&=
\prod_{p\in \cP_f(x)}\ \sum_{\substack{(u,v) \in \cZ_f(p)}} \ep(au),
\end{aligned}
$$
which completes the proof.
\end{proof}

\begin{lemma}\label{lem:Congr}
For any   absolutely irreducible
polynomial
$$f(U,V) \in \Z[U,V],
$$
there is a constant $c > 0$, depending only on $f$ such that
$$
 \prod_{p \in \cP_f(x)} (p + cp^{1/2})\ge  \#\cZ_f(M_f(x))
\ge  \prod_{p \in \cP_f(x)} \max\{1, p - cp^{1/2}\}.
$$
\end{lemma}

\begin{proof}
By the Chinese remaindering theorem we have a bijection between
points of $\cZ_f(M_f(x))$ and tuples of points in $\prod_{p\in
\cP_f(x)} \cZ_f(p)$, see the proof of Lemma~\ref{lem:CRT}. Therefore,
$$
\#\cZ_f(M_f(x))
= \prod_{p \in \Px}\#\cZ_f(p).
$$
As we have noted, the polynomial $f$ remains absolutely
irreducible modulo all sufficiently large prime numbers $p$ by
Ostrowski~\cite{Ost}. Let $c_f$ be the least integer such that $f$ is absolutely
irreducible over $\Z_p$ for all prime numbers $p \ge c_f$ (for explicit bounds on $c_f$ see for
example~\cite{GaoRod,Rup,Zan}).  If $f$ is
absolutely irreducible over $\Z_p$, and $p\ge c_f$
then from the Weil bound we have
\begin{equation}
\label{eq:Weil}
\#\cZ_f(p)= p + O(p^{1/2}),
\end{equation}
see~\cite[Section~VIII.5, Bound~(5.7)]{Lor}.
Furthermore, allowing the implied constant in~\eqref{eq:Weil}
to depend of $f$,
we see that~\eqref{eq:Weil}  trivially holds for all primes $p$.
By definition, for all $p\in \Px$ we have $\#\cZ_f(p)\ge 1$,
which completes the proof.
\end{proof}

\begin{cor}\label{cor:Z-Asymp}
For any   absolutely irreducible
polynomial
$$f(U,V) \in \Z[U,V],
$$
we have
$$
 \#\cZ_f(M_f(x)) = M_f(x)^{1 + o(1)}.
$$
\end{cor}

\begin{proof}
Since
\begin{eqnarray*}
p \exp(c_0 p^{-1/2}) & \ge &p + cp^{1/2} \ge \max\{1, p - cp^{1/2}\}\\
& \ge & p \exp(-c_0 p^{-1/2})
\end{eqnarray*}
for an appropriate  constant $c_0 \ge 0$,
we have
\begin{eqnarray*}
\#\cZ_f(M_f(x)) & =  &M_f(x)  \exp\(O\( \sum_{p \le x} p^{-1/2}\)\)\\
 & =  &  M_f(x)  \exp\(O(x^{1/2}/\log x)\),
\end{eqnarray*}
which together with~\eqref{eq:Mx} implies the result
\end{proof}

Finally we  need the following estimate on the number of
points on a curve with a restricted coordinate which
follows from a  result of Pila~\cite{Pila}
that in turn slightly improves the previous estimate of
Bombieri and Pila~\cite{BombPila}.

\begin{lemma}\label{lem:Equation}
For any   absolutely irreducible
polynomial
$$f(U,V) \in \Z[U,V]
$$
nonlinear in $V$
that is,
$$
\deg_V f \ge 2
$$
the equation
$$f(n,m) = 0, \qquad 0 \le n < N,\ m \in \Z,
$$
has at most $O(N^{1/2+o(1)})$ solutions.
\end{lemma}

\begin{proof} Let $d = \deg f$.
Since $\deg_V f \ge 2$ we see that for any solution to
the above equation we have $m = O(n^{d/2}) = O(N^{d/2})$.
Recalling that  by~\cite{Pila},  an absolutely irreducible polynomial
of degree $d$ has $O(T^{1/d + o(1)})$ solutions in
a box $[0,T]\times [0,T]$ (where the implied constants
depend only on $d$), we derive the result.
\end{proof}

\section{Proof of Theorem~\ref{thm:Distr}}

Let $T_f(N;x)$ be the number of solutions $(n,m)$ to
the congruence
\begin{equation}
\label{eq:congr}
f(n,m)\equiv 0 \pmod {M_f(x)}, \qquad
0 \le n < N,  \ 0 \le m < M_f(x).
\end{equation}

Using~\eqref{eq:Ident}, we write
\begin{eqnarray*}
T_f(N;x)  &= & \sum_{\substack{(n,m) \in \cZ_f(M_f(x))}}\
\sum_{k=0}^{N-1} \frac{1}{M_f(x)} \\
& & \qquad \qquad \qquad \quad
\sum_{-M_f(x)/2 < a \le M_f(x)/2}
\eMx(a(n-k)) \\
& = & \frac{1}{M_f(x)} \sum_{-M_f(x)/2 < a \le M_f(x)/2}\\
& & \qquad \qquad \qquad \quad
\sum_{\substack{(n,m) \in \cZ_f(M_f(x))}}\hspace{-4mm} \eMx(an)
\sum_{k=0}^{N-1}   \eMx(-ak).
\end{eqnarray*}
Separating the term corresponding to
$a = 0$, and recalling~\eqref{eq:Incompl}, we obtain
\begin{equation}
\label{eq:TQR}
T_f(N;x)  - \frac{N}{M_f(x)}\#\cZ_f(M_f(x))  = O\(R_f(N;x)\),
\end{equation}
where, as before, $\cZ_f(q)$ is defined by~\eqref{eq:Zfq}
and
$$
R_f(N;x) =
  \sum_{0 < |a| \le M_f(x)/2}  \frac{1}{|a|}
\left|\sum_{\substack{(n,m) \in \cZ_f(M_f(x))}} \eMx(an)\right|.
$$

To estimate $R_f(N;x)$, for every $d  \mid  M_f(x)$, we collect
together the values of $a$ with the same value $\gcd(a,M_f(x)) = d$
and write them as $a = db$, getting
$$
R_f(N;x) = \sum_{d \mid    M_f(x)}\frac{1}{d}
\sum_{\substack{0 < |b| \le M_f(x)/2d \\ \gcd(b, M_f(x)/d)=1}}
\frac{1}{|b|}
\left|\sum_{\substack{(n,m) \in \cZ_f(M_f(x))}} \eMx(dbn)\right|.
$$

Recalling Lemma~\ref{lem:CRT} and  then estimating
the corresponding exponential sums via Lemma~\ref{lem:Bomb}
for $p \nmid d$ (and using the trivial bound $\#  \cZ_f(p)$
for $p \mid d$),
we deduce
$$
\left|\sum_{\substack{(n,m) \in \cZ_f(M_f(x))}} \eMx(dbn)\right|
\le C^{\pi(x)}  (M_f(x)/d)^{1/2} d = C^{\pi(x)}  (M_f(x) d)^{1/2} ,
$$
where  $C$ is the implied constant of  Lemma~\ref{lem:Bomb} and, as
usual, $\pi(x)$ is the number of primes $p\le x$. Therefore,
using~\eqref{eq:Mx}
\begin{eqnarray*}
R_f(N;x) & \le & C^{\pi(x)} \sum_{d \mid    M_f(x)}\frac{1}{d}
\sum_{\substack{0 < |b| \le M_f(x)/2d \\ \gcd(b, M_f(x)/d)=1}}
\frac{ (M_f(x) d)^{1/2} }{|b|} \\
& = & C^{\pi(x)} M_f(x)^{1/2} \sum_{d \mid    M_f(x)}\frac{1}{d^{1/2}}
\sum_{\substack{0 < |b| \le M_f(x)/2d \\ \gcd(b, M_f(x)/d)=1}}
\frac{ 1}{|b|} \\
& =  & O\( (2C)^{\pi(x)} M_f(x)^{1/2} \log M_f(x) \)
=O\( M_f(x)^{1/2+o(1)}\),
\end{eqnarray*}
as trivially
$$
\sum_{d \mid    M_f(x)}\frac{1}{d^{1/2}}
 \le \sum_{d \mid    M_f(x)} 1 \le 2^{\pi(x)}.
$$
Thus, we see from~\eqref{eq:TQR} that
$$
\left|T_f(N;x)  - \frac{N}{M_f(x)}\#\cZ_f(M_f(x))\right|
\le M_f(x)^{1/2+o(1)}.
$$
Furthermore, by Corollary~\ref{cor:Z-Asymp} we see that
for any fixed $\varepsilon> 0$, taking
$N = \rf{M_f(x)^{1/2+\varepsilon}}$ we obtain
$$
T_f(N;x) = N^{1 + o(1)}.
$$
By the Chinese remaindering theorem we see that for every fixed $n$
there are no more that $d^{\pi(x)}$ solutions $m$ to the congruence
$f(n,m)\equiv 0 \pmod {M_f(x)}$, $0 \le m < M_f(x)$. Thus we have at
least
$$
T_f(N;x) d^{-\pi(x)} =  T_f(N;x)M_f(x)^{o(1)}= N^{1 + o(1)}.
$$
values of $n$ for which the congruence $f(n,m)\equiv 0 \pmod {M_f(x)}$,
has a solution.
Using Lemma~\ref{lem:Equation} we see that
there is a  solution $(n,m)$ to~\eqref{eq:congr}
with $f(n,m) \ne 0$. Since $\varepsilon$ is arbitrary.
Since $\varepsilon$ is arbitrary,
this concludes the proof.

\section{Comments}
\label{sec:Com}

It is easy to see that all implicit constants in our estimates can be
efficiently evaluated. For example, see~\cite{GaoRod,Rup,Zan}
for explicit versions of the Ostrowski theorem.

We remark that besides pseudosquares, in a number of
works~\cite{BLSW,BKPS,KPS} the notion of pseudopowers has been
studied.  Namely, following  Bach, Lukes, Shallit and
Williams~\cite{BLSW}, we say that  an integer $n>0$ is  an {\it
$x$-pseudopower to  base $g$\/} (for a given integer $g$ with
$|g|\ge 2$) if $n$ is not a power of $g$ over the integers but is a
power of $g$  modulo all primes $p\le x$, that is, if for all primes
$p\le x$ there exists an integer $k_p\ge 0$ such that $n\equiv
g^{k_p} \pmod p $.

The notion of pseudopowers naturally extends to elliptic curves.
More precisely, given an elliptic curve $E$ over $\Q$ and a
rational point $P \in E(\Q)$ we say that $Q$ is an
$x$-pseudomultiple of $P$  if for every prime $p \le x$ at which $E$
has good reduction, there is an integer $k_p\ge 0$ so that $ Q
\equiv k_p P \pmod p$ but $Q$ is not of the form $Q = kP+T$ for some
integer $k$ and some torsion point $T$. Clearly $E$ has to be of
rank at least 2 for $x$-pseudomultiple to exist. Obtaining upper
bounds on the canonical height of the smallest pseudopowers is a
natural and challenging questions.

\section*{Acknowledgements.}

The authors are grateful to Joe Silverman for the idea of the
question on $x$-pseudomultiples on elliptic
curves in Section~\ref{sec:Com}; in fact this also
gave us the idea of defining and studying $x$-pseudopoints
on curves.

Both authors were supported in part
by ARC Grant DP0881473.

\end{document}